\documentclass[12pt]{article}
\usepackage{amsfonts}
\usepackage{amsmath,enumerate}

\newtheorem{theorem}{Theorem}

\newtheorem{lemma}[theorem]{Lemma}

\newtheorem{proposition}[theorem]{Proposition}

\newenvironment{proof}[1][Proof]{\noindent\textbf{#1.} }{\ \rule{0.5em}{0.5em}}
\begin{document}

\title{Groups from Class 2 Algebras and the Weil Character}
\author{Christakis A. Pallikaros$^a$ and Harold N. Ward$^b$}

\maketitle

{\vspace{1mm}\par\noindent\footnotesize\it{${}^{a}$~Department of Mathematics and Statistics, University of Cyprus, PO Box 20537,
1678 Nicosia, Cyprus\\
\phantom{${}^{a}$~}{}E-mail: pallikar@ucy.ac.cy\\
${}^{b}$~Department of Mathematics, University of Virginia, Charlottesville, VA 22904, USA\\
\phantom{${}^{b}$~}{}E-mail: hnw@virginia.edu

}}

\begin{abstract}
We investigate the behaviour of the Weil character of the symplectic group on
restriction to subgroups arising from commutative nilpotent algebras of class 2.
We give explicit descriptions of the decomposition of the Weil character when 
restricted to the unipotent radical of the stabilizer of a maximal totally 
isotropic subspace 
and to its centralizer. 
\end{abstract}

\section{Introduction}

Because of their significant properties, Weil representations play an
important role in the study of the representation theory of classical
groups. The characters of Weil representations have been computed by various
authors; see for example \cite{G}, \cite{H}, \cite{Is}, \cite{Tho}. We will
be making use of some explicit results concerning the Weil characters of the
symplectic group as these are obtained in \cite{MWR} via the `theta form'
(see also \cite{Tho}). The approach in~\cite{MWR} is to follow the treatment
for the Weil representation given in~\cite{SP}.

One of the aims of the present work is to follow up in the direction of some
of the investigations in \cite{PW}, \cite{PZ} concerning the restriction of
the Weil characters of symplectic and unitary groups to certain subgroups,
in particular to certain self-centralizing subgroups. In \cite{PZ} an
explicit description of the restriction of the Weil character (respectively,
the sum of the two Weil characters) of the unitary group (respectively,
symplectic group) to centralizers of regular unipotent elements was also
obtained. We remark here that the study of the behaviour of element
centralizers in Weil representations already began in \cite{HZ}.

At this point we introduce some notation. Fix $q$ a power of an odd prime
and consider the Weil representation and its character $\omega$
 for $\mathrm{Sp}%
(2n,q)$, regarded as a matrix group on $V=GF(q)^{2n}$ defined via the
symplectic form $\varphi$ having matrix $%
\begin{bmatrix}
0 & I \\
-I & 0%
\end{bmatrix}%
$. The subgroup $G$ of $\mathrm{Sp}(2n,q)$ to be considered is $G=\langle
-I\rangle\times B$ where $B$ is the set of matrices $g_Q=%
\begin{bmatrix}
I & 0 \\
Q & I%
\end{bmatrix}%
$, with $Q$ a symmetric matrix. So $B$ is the unipotent radical of the
stabilizer in $\mathrm{Sp}(2n,q)$ of a maximal totally isotropic subspace of
$V$, and $G$, which is the centralizer of $B$ in $\mathrm{Sp}(2n,q)$, is a
maximal Abelian subgroup of $\mathrm{Sp}(2n,q)$. Also note that the matrices
$g_Q-I$ with $g_Q\in B$ belong to a nilpotent subalgebra of the full matrix
algebra of $2n\times 2n$ matrices over $GF(q)$ of class 2 (compare with the
discussion in \cite{PW}).

The restriction $\omega|G$ is known to be multiplicity free. The main result
of the paper is Theorem~\ref{TheoremMain}, where we describe explicitly the
irreducible characters of $G$ (respectively, $B$) appearing in the
decomposition of $\omega|G$ (respectively, $\omega|B$). As a by-product, in
Section~\ref{ConfDecomp} we use the expression for the decomposition $%
\omega|B$ in order to obtain alternative derivations for the number of
solutions of the equation $Q(x)=\alpha$ with $\alpha\in GF(q)$, $x\in
GF(q)^n $ and $Q$ a quadratic form of $GF(q)^n$.

One can easily observe that a subgroup $H$ of $G$ with a multiplicity free
restriction $\omega |H$ satisfies $|H|\geq 2q^{n}$. In Section~\ref%
{SecMinWeil-freeSubgr} we show that this bound is actually attained and we
give a construction of such a subgroup $H$ with $|H|=2q^{n}$ and $\omega |H$
being multiplicity free. 
Finally, we show how an evaluation
concerning the values of the Weil character on the elements of this subgroup
$H$ provides a link with the Davenport--Hasse theorem on lifted Gauss sums.

The decompositions in our main theorem \ref{TheoremMain} corroborate results
in the preprint by Gurevich and Howe \cite{GH} (which appeared after we
finished our manuscript), especially those in Section~2.

\section{Preliminaries\label{SecPre}}

We consider the Weil representation and its character $\omega $ for $\mathrm{%
Sp}(2n,q)$, $q$ a power of an odd prime, as a matrix group on $V=GF(q)^{2n}$
(with right action). 
We will be following the notation about $\omega$ in~\cite{SP}, \cite{MWR} and
use some explicit results in~\cite{MWR} about its values.
The symplectic form $\varphi $ has matrix $%
\begin{bmatrix}
0 & I \\
-I & 0%
\end{bmatrix}%
$. Writing members of $V$ as pairs $(x,y)$, with $x,y\in GF(q)^{n}$, one has%
\begin{equation*}
\varphi ((x_{1},y_{1}),(x_{2},y_{2}))=x_{1}y_{2}^{T}-y_{1}x_{2}^{T},
\end{equation*}%
the superscript $T$ standing for transpose.

The subgroup $G$ of $\mathrm{Sp}(2n,q)$ to be considered is $G=\left\langle
-I\right\rangle \times B$, where $B$ is the set of matrices $g_{Q}=%
\begin{bmatrix}
I & 0 \\
Q & I%
\end{bmatrix}%
$. Such a matrix is in $\mathrm{Sp}(2n,q)$ just when $Q^{T}=Q$; that is, $Q$
is symmetric (\textquotedblleft $Q$\textquotedblright\ emphasizes the
associated quadratic form $Q(v)=vQv^{T}$ on $GF(q)^{n}$). Let $\boldsymbol{S}%
_{n}$ be the space of $n\times n$ symmetric matrices over $GF(q)$. There are
$q^{n(n+1)/2}$ such matrices.

The following theorem follows from results in \cite{GMST}, but we present a
proof done in the spirit of \cite{PW}.

\begin{theorem}
\label{TheoremGWeilfree}$G$ is \textbf{Weil-free}: the irreducible
constituents of $\omega |G$ appear with multiplicity 1.
\end{theorem}

\begin{proof}
We use the \textbf{orbit criterion}: an Abelian subgroup $H$ of $\mathrm{Sp}%
(2n,q)$ is Weil-free exactly when the number of orbits of $H$ on $V$ is $%
q^{n}$ \cite{PW}. To count the orbits of $G$, recall that $(\#$ orbits$%
)=\left\vert G\right\vert ^{-1}\displaystyle\sum_{v\in V}\left\vert
G_{v}\right\vert $, $G_{v}$ the stabilizer of $v$. Notice that the matrices $%
-g_{Q}=%
\begin{bmatrix}
-I & 0 \\
-Q & -I%
\end{bmatrix}%
$ in $-B$ fix only $0$. Since $(x,y)^{g_{Q}}=(x,y)%
\begin{bmatrix}
I & 0 \\
Q & I%
\end{bmatrix}%
=(x+yQ,y)$, $v=(x,y)$ is fixed exactly when $yQ=0$. We look at three cases:

$v=(0,0)$: then $\left\vert G_{(0,0)}\right\vert =\left\vert G\right\vert
=2q^{n(n+1)/2}$.

$v=(x,0),x\neq 0$: all members of $B$ fix $v$ and $\left\vert
G_{(x,0)}\right\vert =q^{n(n+1)/2}$. There are $q^{n}-1$ such~$v$.

$v=(x,y),y\neq 0$: $v$ is fixed by the elements $g_{Q}$ with $yQ=0$. We can
set up such $Q$ by thinking of it as a quadratic form with $y$ in the
radical. Write $GF(q)^{n}=\left\langle y\right\rangle \oplus W$. Then $Q$
can be given by taking a form on $W$ and extending it by $0$ on $%
\left\langle y\right\rangle $. That gives $\left\vert G_{(x,y)}\right\vert
=q^{n(n-1)/2}$, the number of choices for the form on $W$. There are $%
q^{n}(q^{n}-1)$ of these~$v$.

\noindent Thus for $\displaystyle\sum_{v\in V}\left\vert G_{v}\right\vert $
we get%
\begin{eqnarray*}
\displaystyle\sum_{v\in V}\left\vert G_{v}\right\vert
&=&2q^{n(n+1)/2}+(q^{n}-1)\times q^{n(n+1)/2}+q^{n}(q^{n}-1)\times
q^{n(n-1)/2} \\
&=&q^{n(n+1)/2}\left\{ 2+q^{n}-1+q^{n}-1\right\} \\
&=&q^{n}\times 2q^{n(n+1)/2}.
\end{eqnarray*}%
So $\left\vert G\right\vert ^{-1}\displaystyle\sum_{v\in V}\left\vert
G_{v}\right\vert =q^{n}$, as needed.
\end{proof}

\section{Irreducible characters of $G$\label{SecIrredG}}

To describe the characters of $G$, let $\psi $ be the \textbf{canonical
additive character} of $GF(q)$, as used in \cite{SP} (the terminology is
that of \cite[p. 190]{LN}): $\psi (\alpha )=e^{(2\pi i/p)\mathrm{tr}(\alpha
)}$, where $\mathrm{tr}$ is the trace function $GF(q)\rightarrow GF(p)$, $p$
the prime dividing $q$. Each linear character of the additive group of $%
GF(q) $ is given by $\alpha \rightarrow \psi (\beta \alpha )$, $\beta \in
GF(q)$ \cite[Theorem 5.7]{LN} (this is equivalent to the nondegeneracy of
the trace form $(\alpha ,\beta )\rightarrow \mathrm{tr}(\alpha \beta )$). In
what follows, $\chi $ is the quadratic character on $GF(q)^{\#}$ (the
nonzero elements) and $\delta =\chi (-1)$. In \cite{SP}, $\rho $ was defined
as $\sum_{\alpha \in GF(q)}\psi (\alpha ^{2})$; we also have
\begin{equation}
\rho =\sum_{\beta \neq 0}\chi (\beta )\psi (\beta ),  \label{rho}
\end{equation}%
a Gaussian sum \cite[Chapter 5, Section 2]{LN}; $\rho ^{2}=\delta q$. If $%
Q\in \boldsymbol{S}_{n}$, diagonalize $Q$ and let
\begin{equation*}
\Delta (Q)=\chi (\text{product of nonzero diagonal entries of }Q).
\end{equation*}%
That is, if we write $GF(q)^{n}$ as $\mathrm{rad}(Q)\oplus W$, $\Delta (Q)$
is $\chi (\det (Q|W))$; $Q|W$ is the \textbf{nonsingular part} of $Q$. If $Q$
has even rank $2k$, we call $Q$ \textbf{hyperbolic} or \textbf{elliptic}
according as $Q|W$ is hyperbolic or elliptic. If $Q|W$ is hyperbolic, then $%
W $ is the orthogonal sum of hyperbolic planes, and $\det Q|W=(-1)^{k}$.
Thus $\Delta (Q)=\delta ^{k}$. If $Q|W$ is elliptic, then $\Delta
(Q)=-\delta ^{k}$.

\begin{lemma}
The irreducible characters of $B$ are the functions $\lambda _{S}$ given by%
\begin{equation*}
\lambda _{S}\left( g_{Q}\right) =\psi (\mathrm{Tr}(SQ)),
\end{equation*}%
where $S\in \boldsymbol{S}_{n}$ and $\mathrm{Tr}$ is the matrix trace. Each $%
\lambda _{S}$ extends to two irreducible characters $\lambda _{S}^{\pm }$ of
$G$ by the formula%
\begin{eqnarray}
\lambda _{S}^{\pm }\left( g_{Q}\right) &=&\psi (\mathrm{Tr}(SQ))
\label{lambdavalues} \\
\lambda _{S}^{\pm }\left( -g_{Q}\right) &=&\pm \psi (\mathrm{Tr}(SQ)),
\notag
\end{eqnarray}%
with the signs in the last equation matching on the two sides.
\end{lemma}

\begin{proof}
That this formula does give all the linear characters of $B$ follows from
the fact that the trace form $(S,Q)\rightarrow \mathrm{Tr}(SQ)$ on $%
\boldsymbol{S}_{n}$ is nondegenerate. That, in turn, can be seen as follows:
suppose that $\mathrm{Tr}(SQ)=0$ for all $Q\in \boldsymbol{S}_{n}$. Take a
basis for for $GF(q)^{n}$ that makes $S$ diagonal. Suppose that $\zeta $ is
a nonzero diagonal entry of $S$. Choose $Q$ to have 1 at that position and 0
elsewhere. Then $\mathrm{Tr}(SQ)=\zeta $; so $\zeta =0$ after all. Thus $S=0$%
. So these characters $\lambda _{S}$ are all distinct; and since there is
the correct number of them, they give all the characters of $B$. Then for
the characters of $G$, we use the direct product decomposition $%
G=\left\langle -I\right\rangle \times B$ to write them as claimed.
\end{proof}

\section{Values of $\protect\omega $ on $G$\label{SecwVal}}

For the values of $\omega $, we use results from \cite{MWR}. A member $%
-g_{Q} $ of $-B$ has $-g_{Q}-1$ invertible, with diagonal entries $-2$, and
\cite[Section 6.5]{MWR} gives%
\begin{equation}
\omega \left( -g_{Q}\right) =\delta ^{n}\chi (\det (-g_{Q}-1))=\delta
^{n}\chi ((-2)^{2n})=\delta ^{n}.  \label{w(-1)}
\end{equation}%
As for $g_{Q}$, \cite[Theorem 6.7]{MWR} implies that $\omega (g)=q^{n}\rho
^{-\dim V^{g-1}}\chi (\det \Theta _{g}).$ Here $\Theta _{g}$ is given in
\cite[Definition 3.3]{MWR}: it is the form defined on $V^{g-1}$ by $\Theta
_{g}(u^{g-1},v^{g-1})=\varphi (u^{g-1},v)$. We have $(x,y)^{g-1}=(yQ,0)$. So
\begin{eqnarray*}
\Theta _{g}((x_{1},y_{1})^{g-1},(x_{2},y_{2})^{g-1}) &=&\varphi
((x_{1},y_{1})^{g-1},(x_{2},y_{2})) \\
&=&\varphi ((y_{1}Q,0),(x_{2},y_{2})) \\
&=&y_{1}Qy_{2}^{T}.
\end{eqnarray*}%
It follows that $\chi (\det \Theta _{g})=\Delta (Q)$. Thus

\begin{proposition}
The values of $\omega $ on $G$ are given by%
\begin{eqnarray}
\omega \left( g_{Q}\right) &=&q^{n}\rho ^{-\mathrm{rank}Q}\Delta (Q)
\label{w(1)} \\
\omega \left( -g_{Q}\right) &=&\delta ^{n}.
\end{eqnarray}%
If $Q=0$, $\Delta (Q)$ is artificially taken to be $1$.
\end{proposition}

\section{Character multiplicities in $\protect\omega |G$\label{SecCharMult}}

The multiplicity (which is 0 or 1) of a linear character $\lambda $ in $%
\omega |G$ is $\left\vert G\right\vert ^{-1}\sum_{g\in G}\omega (g)\overline{%
\lambda (g)}.$ With $\lambda =\lambda _{S}^{\pm }$, we get%
\begin{equation*}
\sum_{g\in G}\omega (g)\overline{\lambda _{S}^{\pm }(g)}=\sum_{Q\in
\boldsymbol{S}_{n}}q^{n}\rho ^{-\mathrm{rank}Q}\Delta (Q)\psi (-\mathrm{Tr}%
(SQ))\pm \sum_{Q\in \boldsymbol{S}_{n}}\delta ^{n}\psi (-\mathrm{Tr}(SQ))
\end{equation*}%
(again, the sign on the right matches the sign in $\lambda _{S}^{\pm }$).
The second summation is 0 if $S\neq 0$, and $\delta ^{n}q^{n(n+1)/2}$ if $%
S=0 $.

For further computations, we need some standard group-order formulas. They
are taken from \cite{T}.%
\begin{equation}
\begin{tabular}{cc}
Group & Order \\
$GL(l,q)$ & $q^{l(l-1)/2}\prod\limits_{i=1}^{l}(q^{i}-1)$ \\
$O^{+}(2k,q)$ & $2q^{k(k-1)}(q^{k}-1)\prod\limits_{i=1}^{k-1}(q^{2i}-1)$ \\
$O^{-}(2k,q)$ & $2q^{k(k-1)}(q^{k}+1)\prod\limits_{i=1}^{k-1}(q^{2i}-1)$ \\
$O(2k+1,q)$ & $2q^{k^{2}}\prod\limits_{i=1}^{k}(q^{2i}-1).$%
\end{tabular}
\label{orders}
\end{equation}%
We also need the $q$-binomial coefficient 
${n \brack r}_{q}$ that gives the number of $r$-dimensional subspaces of $%
GF(q)^{n}$. By duality, 
${n \brack r}_{q}={n \brack n-r}_{q}$.

\subsection{$S=0$}

When $S=0$,%
\begin{equation*}
(\omega ,\lambda _{0}^{\pm })_{G}=\frac{1}{2q^{n(n-1)/2}}\sum_{Q\in
\boldsymbol{S}_{n}}\rho ^{-\mathrm{rank}Q}\Delta (Q)\pm \frac{\delta ^{n}}{2}%
,
\end{equation*}%
since $\left\vert G\right\vert =2q^{n(n+1)/2}$. We conclude that the first
term must be $1/2$:%
\begin{equation}
\sum_{Q\in \boldsymbol{S}_{n}}\rho ^{-\mathrm{rank}Q}\Delta (Q)=q^{n(n-1)/2}.
\label{sum1}
\end{equation}%
We shall elaborate on this in Section \ref{Secqbinom}. Thus we have

\begin{proposition}
The multiplicity of $\lambda _{0}^{\pm }$ in $\omega $ is%
\begin{equation*}
(\omega ,\lambda _{0}^{\pm })_{G}=\frac{1\pm \delta ^{n}}{2}.
\end{equation*}
\end{proposition}

\subsection{$S\neq 0$}

Here%
\begin{equation*}
\sum_{g\in G}\omega (g)\overline{\lambda _{S}^{\pm }(g)}=\sum_{Q\in
\boldsymbol{S}_{n}}q^{n}\rho ^{-\mathrm{rank}Q}\Delta (Q)\psi (-\mathrm{Tr}%
(SQ)).
\end{equation*}%
It follows that $\lambda _{S}^{+}$ and $\lambda _{S}^{-}$ appear with the
same multiplicity in $\omega |G$. Now note that $\mathrm{Tr}(M^{T}SMQ)=%
\mathrm{Tr}(SMQM^{T})$. So with $M$ nonsingular, $\lambda _{S}^{\pm }$ and $%
\lambda _{M^{T}SM}^{\pm }$ also have the same multiplicity in $\omega |G$.
Suppose that $\mathrm{\mathrm{rank}}S=r$. The number of members of $%
\boldsymbol{S}_{n}$ congruent to $S$ (that is, of the form $M^{T}SM$) is%
\begin{equation*}
{n \brack n-r}_{q} \times \frac{\left\vert GL(r,q)\right\vert }{\left\vert
O(S)\right\vert }= {n \brack r}_{q} \times \frac{\left\vert
GL(r,q)\right\vert }{\left\vert O(S)\right\vert },
\end{equation*}%
where $O(S)$ is the orthogonal group for $S$. We want to show that if $r>1$,
then this number is more than $(q^{n}-1)/2$. That will imply that the only $%
S\neq 0$ that can appear in the characters in $\omega |G$ are the ones of
rank 1; there are $q^{n}-1$ of these \cite[Theorem 13.2.47]{HFF}. Again we
separate by parity.

\begin{itemize}
\item $r=2k$: Then%
\begin{equation*}
\left\vert O(S)\right\vert \leq \left\vert O^{-}(2k,q)\right\vert
=2q^{k(k-1)}(q^{k}+1)\prod\limits_{i=1}^{k-1}(q^{2i}-1),
\end{equation*}%
from (\ref{orders})\newline
So%
\begin{eqnarray*}
{n \brack r}_{q} \times \frac{\left\vert GL(r,q)\right\vert }{\left\vert
O(S)\right\vert } &\geq & {n \brack 2k}_{q} \times \frac{\left\vert
GL(2k,q)\right\vert }{\left\vert O^{-}(2k,q)\right\vert } \\
&=&\frac{\prod\limits_{j=n-2k+1}^{n}(q^{j}-1)}{\prod%
\limits_{j=1}^{2k}(q^{j}-1)}\times \frac{q^{k(2k-1)}\prod%
\limits_{j=1}^{2k}(q^{j}-1)}{2q^{k(k-1)}(q^{k}+1)\prod%
\limits_{i=1}^{k-1}(q^{2i}-1)} \\
&=&\frac{\prod\limits_{j=n-2k+1}^{n}(q^{j}-1)}{2(q^{k}+1)\prod%
\limits_{i=1}^{k-1}(q^{2i}-1)}\times q^{k^{2}} \\
&\geq &\frac{\prod\limits_{j=n-2k+1}^{n}(q^{j}-1)}{2(q^{k}+1)\prod%
\limits_{i=1}^{k-1}q^{2i}}\times q^{k^{2}} \\
&=&\frac{\prod\limits_{j=n-2k+1}^{n}(q^{j}-1)}{2(q^{k}+1)}\times q^{k}.
\end{eqnarray*}%
We would like this to be more than $(q^{n}-1)/2$, that is,%
\begin{equation*}
\frac{\prod\limits_{j=n-2k+1}^{n}(q^{j}-1)}{2(q^{k}+1)}\times q^{k}>\frac{%
q^{n}-1}{2}.
\end{equation*}%
We need%
\begin{equation*}
\prod\limits_{j=n-2k+1}^{n-1}(q^{j}-1)>1+\frac{1}{q^{k}}.
\end{equation*}%
Since $n-1\geq n-2k+1$, from $k\geq 1$, and $n-2k+1\geq 1$, the product is
nonempty and its smallest factor is at least 2. Thus the inequality holds.

\item $r=2k+1$: This time, again with the appropriate formula from (\ref%
{orders}) filled in,%
\begin{eqnarray*}
{n \brack r}_{q} \times \frac{\left\vert GL(r,q)\right\vert }{\left\vert
O(S)\right\vert } &=&\frac{\prod\limits_{j=n-2k}^{n}(q^{j}-1)}{%
\prod\limits_{j=1}^{2k+1}(q^{j}-1)}\times \frac{q^{k(2k+1)}\prod%
\limits_{j=1}^{2k+1}(q^{j}-1)}{2q^{k^{2}}\prod\limits_{j=1}^{k}(q^{2j}-1)} \\
&=&\frac{\prod\limits_{j=n-2k}^{n}(q^{j}-1)}{2\prod%
\limits_{j=1}^{k}(q^{2j}-1)}\times q^{k^{2}+k} \\
&\geq &\frac{\prod\limits_{j=n-2k}^{n}(q^{j}-1)}{2\prod%
\limits_{j=1}^{k}q^{2j}}\times q^{k^{2}+k} \\
&=&\frac{\prod\limits_{j=n-2k}^{n}(q^{j}-1)}{2}.
\end{eqnarray*}%
We would also like this to be greater than $(q^{n}-1)/2$, and as long as $%
k>0 $, it is.
\end{itemize}

\section{The decomposition of $\protect\omega |G$\label{SecDecomp}}

Collecting the results of the preceding section gives

\begin{equation*}
\omega |G=\frac{1+\delta ^{n}}{2}\lambda _{0}^{+}+\frac{1-\delta ^{n}}{2}%
\lambda _{0}^{-}+\sum_{\mathrm{\mathrm{rank}}S=1}\left( \frac{1}{%
2q^{n(n-1)/2}}\sum_{Q\in \boldsymbol{S}_{n}}\rho ^{-\mathrm{rank}Q}\Delta
(Q)\psi (-\mathrm{Tr}(SQ))\right) (\lambda _{S}^{+}+\lambda _{S}^{-}).
\end{equation*}%
We still need to determine which congruence class of symmetric matrices $S$
of rank 1 actually appears in the decomposition. (There are two such
classes, corresponding to $\Delta (S)=1$ and $\Delta (S)=-1$. Each class has
$(q^{n}-1)/2$ members.) To do so, we examine the characters on a small
subgroup of $G$.

Let $M$ be the $n\times n$ symmetric matrix with $M_{11}=1$ and all other
entries $0$; $\Delta (M)=1$. Let $H$ be the subgroup consisting of the
matrices $h_{\alpha }=%
\begin{bmatrix}
I & 0 \\
\alpha M & I%
\end{bmatrix}%
$, $\alpha \in GF(q)$. Then $\omega (h_{\alpha })=q^{n}$ if $\alpha =0$, and
$\omega (h_{\alpha })=q^{n}\rho ^{-1}\chi (\alpha )$ if $\alpha \neq 0$, by (%
\ref{w(1)}). Moreover, $\lambda _{\beta M}^{\pm }(h_{\alpha })=\psi
(\alpha\beta ) $, for $\beta \neq 0$. It follows that%
\begin{eqnarray*}
(\omega ,\lambda _{\beta M}^{\pm })_{H} &=&q^{-1}\left\{ q^{n}+\sum_{\alpha
\neq 0}q^{n}\rho ^{-1}\chi (\alpha )\psi (-\alpha \beta )\right\} \\
&=&q^{n-1}\left\{ 1+\delta \rho ^{-1}\chi (\beta )\sum_{\alpha \neq 0}\chi
(-\beta \alpha )\psi (-\beta \alpha )\right\} \\
&=&q^{n-1}(1+\chi (\beta )\delta ),
\end{eqnarray*}%
by (\ref{rho}). So $\lambda _{\beta M}^{\pm }$ appears in $\omega |H$ just
when $\chi (\beta )=\delta $. This implies the following:

\begin{theorem}
\label{TheoremMain}
\begin{equation}
\omega |G=\frac{1+\delta ^{n}}{2}\lambda _{0}^{+}+\frac{1-\delta ^{n}}{2}%
\lambda _{0}^{-}+\sum_{\substack{ \mathrm{\mathrm{rank}}S=1  \\ \Delta
(S)=\delta }}(\lambda _{S}^{+}+\lambda _{S}^{-}).  \label{w|G}
\end{equation}%
In particular,%
\begin{equation}
\omega |B=\lambda _{0}+2\sum_{\substack{ \mathrm{\mathrm{rank}}S=1  \\ %
\Delta (S)=\delta }}\lambda _{S}.  \label{w|B}
\end{equation}
\end{theorem}

The Weil character is the sum of two irreducible characters, $\omega _{+}$,
of degree $(q^{n}+1)/2$, and $\omega _{-}$, of degree $(q^{n}-1)/2$. Their
values at $-I$ are $\omega _{\pm }(-I)=\pm \delta ^{n}(q^{n}\pm 1)/2$ \cite[%
Section 6]{MWR}, the signs all matching. Observing the eigenvalues of $-I$
in the corresponding representations, we can write that%
\begin{eqnarray*}
\omega _{+}|G &=&\frac{1+\delta ^{n}}{2}\lambda _{0}^{+}+\frac{1-\delta ^{n}%
}{2}\lambda _{0}^{-}+\sum_{\substack{ \mathrm{\mathrm{rank}}S=1  \\ \Delta
(S)=\delta }}\lambda _{S}^{\delta ^{n}} \\
\omega _{-}|G &=&\sum_{\substack{ \mathrm{\mathrm{rank}}S=1  \\ \Delta
(S)=\delta }}\lambda _{S}^{-\delta ^{n}}.
\end{eqnarray*}%
As an immediate consequence we get that both $\omega _{+}|B$ and $\omega
_{-}|B$ are multiplicity free.

\section{Confirmation of the $\protect\omega |B$ decomposition\label%
{ConfDecomp}}

Recall that $g_{Q}=%
\begin{bmatrix}
I & 0 \\
Q & I%
\end{bmatrix}%
$. By (\ref{w(1)}), $\omega (g_{Q})=q^{n}\rho ^{-\mathrm{\mathrm{rank}}%
Q}\Delta (Q)$. Let $\mathrm{rank}Q=r$. Then (\ref{w|B}) gives%
\begin{equation}
\omega (g_{Q})=q^{n}\rho ^{-r}\Delta (Q)=1+2\sum_{\substack{ \mathrm{\mathrm{%
rank}}S=1  \\ \Delta (S)=\delta }}\psi (\mathrm{Tr}(SQ)).  \label{wgQ}
\end{equation}%
If $x\in GF(q)^{n}$, $x\neq 0$, then $x^{T}x$ is a rank 1 symmetric matrix.
Two such products $x^{T}x$ and $y^{T}y$ are equal just when $y=\pm x$. For $%
x=(\xi _{1},\ldots ,\xi _{n})$, the diagonal entries of $x^{T}x$ are the $%
\xi _{i}^{2}$. Thus since at least one is nonzero, $\Delta (x^{T}x)=1$. So
all symmetric $n\times n$ rank 1 matrices can be written as $x^{T}x$ ($x\neq
0$) or $\nu x^{T}x$, where $\nu $ is a fixed nonsquare in $GF(q)$. (As
mentioned above, there are $(q^{n}-1)/2$ matrices of each type.) We can
rewrite (\ref{wgQ}) as follows:%
\begin{equation}
\omega (g_{Q})=q^{n}\rho ^{-r}\Delta (Q)=\sum_{x}\left\{ \frac{1+\delta }{2}%
\psi (\mathrm{Tr}(x^{T}xQ))+\frac{1-\delta }{2}\psi (\mathrm{Tr}(\nu
x^{T}xQ))\right\} .  \label{wgQx}
\end{equation}%
The factors $(1\pm \delta )/2$ pick out the $S$ with $\Delta (S)=\delta $
and adjust for the fact that each $S$ appears twice; the 1 on the right in (%
\ref{wgQ}) comes from $x=0$. (Recall also that we have set $\Delta (0)=1$.)

Now%
\begin{equation*}
\mathrm{Tr}(x^{T}xQ)=\mathrm{Tr}(xQx^{T})=\mathrm{Tr}(Q(x))=Q(x),
\end{equation*}%
so the preceding formula becomes%
\begin{equation}
q^{n}\rho ^{-r}\Delta (Q)=\sum_{x}\left\{ \frac{1+\delta }{2}\psi (Q(x))+%
\frac{1-\delta }{2}\psi (\nu Q(x))\right\} .  \label{sig}
\end{equation}%
Let $\sigma $ denote the sum. To evaluate $\sigma $ we need the number of
times $Q(x)=\alpha $ for $\alpha \in GF(q)$. These counts for nonzero $%
\alpha $ depend only on whether $\alpha $ is a square or not, since $Q(\beta
x)=\beta ^{2}Q(x)$. Let $0$ be taken on $Z$ times (including $Q(0)=0$); a
given nonzero square $S$ times; and a given nonsquare $N$ times. Then $%
Z+(S+N)(q-1)/2=q^{n}$.

We also need the character sums%
\begin{equation}
\sum_{\alpha \neq 0\text{ square}}\psi (\alpha )=\frac{\rho -1}{2}\text{ and
}\sum_{\alpha \text{ nonsquare}}\psi (\alpha )=\frac{-\rho -1}{2},
\label{rho2}
\end{equation}%
which follow from%
\begin{equation*}
\sum_{\alpha \neq 0}\psi (\alpha )=-1,\;\sum_{\alpha \neq 0\text{ square}%
}\psi (\alpha )-\sum_{\alpha \text{ nonsquare}}\psi (\alpha )=\rho .
\end{equation*}%
Collect terms in $\sigma $ according to $Q(x)$:%
\begin{eqnarray*}
\sigma &=&Z\left\{ \frac{1+\delta }{2}+\frac{1-\delta }{2}\right\} +S\left\{
\frac{1+\delta }{2}\frac{\rho -1}{2}+\frac{1-\delta }{2}\frac{-\rho -1}{2}%
\right\} \\
&&+N\left\{ \frac{1+\delta }{2}\frac{-\rho -1}{2}+\frac{1-\delta }{2}\frac{%
\rho -1}{2}\right\} \\
&=&Z-\frac{S+N}{2}+\delta \rho \frac{S-N}{2}.
\end{eqnarray*}%
Then put $Z=q^{n}-(S+N)(q-1)/2$ to obtain%
\begin{equation}
\sigma =q^{n}-\frac{q(S+N)}{2}+\delta \rho \frac{S-N}{2}.  \label{sig2}
\end{equation}

Now we need $S$ and $N$. If $Q_{0}$ is the nonsingular part of $Q$, the
counts for $Q$ are those for $Q_{0}$ multiplied by $q^{n-r}$. Here are these
numbers for $Q$, obtained from \cite[Theorems 6.26 and 6.27]{LN} for $Q_{0}$:%
\begin{equation}
\begin{tabular}{ccc}
& $S$ & $N$ \\
$r$ even: & $q^{n-1}-\delta ^{r/2}\Delta (Q)q^{n-r/2-1}$ & $q^{n-1}-\delta
^{r/2}\Delta (Q)q^{n-r/2-1}=S$ \\
$r$ odd: & $q^{n-1}+\delta ^{(r-1)/2}q^{n-(r+1)/2}\Delta (Q)$ & $%
q^{n-1}-\delta ^{(r-1)/2}q^{n-(r+1)/2}\Delta (Q)$.%
\end{tabular}
\label{SN}
\end{equation}%
Substituting into (\ref{sig2}), we obtain simplifications corresponding to
the parity of $r$.

\begin{itemize}
\item $r$ even: then $S=N$ and%
\begin{eqnarray*}
\sigma &=&q^{n}-\frac{q}{2}(2q^{n-1}-2\delta ^{r/2}\Delta (Q)q^{n-r/2-1}) \\
&=&\delta ^{r/2}\Delta (Q)q^{n-r/2}.
\end{eqnarray*}%
Since $q^{r/2}=\delta ^{r/2}\rho ^{r}$, this is correctly $q^{n}\rho
^{-r}\Delta (Q)$.

\item $r$ odd: then%
\begin{eqnarray*}
\sigma &=&q^{n}-\frac{q}{2}2q^{n-1}+\delta ^{(r-1)/2}q^{n-(r+1)/2}\Delta
(Q)\delta \rho \\
&=&\delta ^{(r+1)/2}q^{n}q^{-(r+1)/2}\rho \Delta (Q) \\
&=&\delta ^{(r+1)/2}q^{n}\delta ^{(r+1)/2}\rho ^{-r-1}\rho \Delta (Q) \\
&=&q^{n}\rho ^{-r}\Delta (Q),
\end{eqnarray*}%
again correct.
\end{itemize}

\subsection{The $S,N$ formulas\label{SecSN}}

In point of fact, the formulas for $S$ and $N$ follow from those for $\omega
|B$. Combining (\ref{wgQ}) and (\ref{sig}) gives%
\begin{equation}
q^{n}\rho ^{-r}\Delta (Q)=q^{n}-\frac{q(S+N)}{2}+\delta \rho \frac{S-N}{2}.
\label{wQ}
\end{equation}%
For a needed second equation, let $Q^{\prime }$ be $Q$ scaled by a
nonsquare. Then the counts for $Q^{\prime }$ are $S^{\prime }=N$ and $%
N^{\prime }=S$, and $\Delta (Q^{\prime })=(-1)^{r}\Delta (Q)$. Formula (\ref%
{wQ}) for $Q^{\prime }$ reads%
\begin{equation}
q^{n}\rho ^{-r}(-1)^{r}\Delta (Q)=q^{n}-\frac{q(S+N)}{2}+\delta \rho \frac{%
N-S}{2}.  \label{wQ'}
\end{equation}%
Solving (\ref{wQ}) and (\ref{wQ'}) for $S$ and $N$ produces%
\begin{eqnarray*}
S &=&q^{n-1}-\frac{1+(-1)^{r}}{2}q^{n-1}\rho ^{-r}\Delta (Q)+\frac{1-(-1)^{r}%
}{2}q^{n}\rho ^{-r-1}\delta \Delta (Q) \\
N &=&q^{n-1}-\frac{1+(-1)^{r}}{2}q^{n-1}\rho ^{-r}\Delta (Q)-\frac{1-(-1)^{r}%
}{2}q^{n}\rho ^{-r-1}\delta \Delta (Q),
\end{eqnarray*}%
and then%
\begin{eqnarray*}
Z &=&q^{n}-(S+N)\frac{q-1}{2} \\
&=&q^{n-1}+\frac{1+(-1)^{r}}{2}q^{n-1}(q-1)\rho ^{-r}\Delta (Q).
\end{eqnarray*}%
These are uniform expressions (perhaps not obvious in traditional
derivations!) which give (\ref{SN}) on taking $r$ even or odd.

Incidentally, when $\delta =-1$ or $q$ is not a square, $\rho $ is not
rational. In that case, the formulas follow from equating coefficients in (%
\ref{wQ}) for the field $\mathbb{Q}(\rho )=\mathbb{Q}+\mathbb{Q\rho }$ and
again solving for $S$ and $N$.

\section{A $q$-binomial identity\label{Secqbinom}}

Recall equation (\ref{sum1}):%
\begin{equation*}
\sum_{Q\in \boldsymbol{S}_{n}}\rho ^{-\mathrm{rank}Q}\Delta (Q)=q^{n(n-1)/2}.
\end{equation*}%
If $\mathrm{rank}Q$ is odd, then with $\nu $ a nonsquare in $GF(q)$, $\Delta
(\nu Q)=-\Delta (Q)$, and the terms for $Q$ and $\nu Q$ in the sum cancel.
Thus as $\rho ^{-2k}=\delta ^{k}q^{-k}$,%
\begin{equation}
q^{n(n-1)/2}=\sum_{Q\in \boldsymbol{S}_{n}}\rho ^{-\mathrm{rank}Q}\Delta
(Q)=\sum_{k=0}^{\left\lfloor n/2\right\rfloor }\sum_{\substack{ Q\in
\boldsymbol{S}_{n}  \\ \mathrm{rank}Q=2k}}\delta ^{k}q^{-k}\Delta (Q).
\label{sum2}
\end{equation}%
Let $\mathrm{rank}Q=2k$. As pointed out in Section \ref{SecIrredG}, if $Q$
is hyperbolic, then $\Delta (Q)=\delta ^{k}$; and if $Q$ is elliptic, then $%
\Delta (Q)=-\delta ^{k}$.

Now we can give a specific formula for $\sum\limits_{\substack{ Q\in
\boldsymbol{S}_{n}  \\ \mathrm{rank}Q=2k}}\delta ^{k}q^{-k}\Delta (Q)$. The
number of terms with $Q$ hyperbolic is
\begin{equation*}
{n \brack 2k}_{q} \times \frac{\left\vert GL(2k,q)\right\vert }{\left\vert
O^{+}(2k,q)\right\vert },
\end{equation*}%
and the number with $Q$ elliptic is%
\begin{equation*}
{n \brack 2k}_{q} \times \frac{\left\vert GL(2k,q)\right\vert }{\left\vert
O^{-}(2k,q)\right\vert }.
\end{equation*}%
Thus%
\begin{equation*}
\sum\limits_{\substack{ Q\in \boldsymbol{S}_{n}  \\ \mathrm{rank}Q=2k}}%
\delta ^{k}q^{-k}\Delta (Q)=\delta ^{k}q^{-k} {n \brack 2k}_{q} \times
\left\vert GL(2k,q)\right\vert \times \delta ^{k}\left( \frac{1}{\left\vert
O^{+}(2k,q)\right\vert }-\frac{1}{\left\vert O^{-}(2k,q)\right\vert }\right)
;
\end{equation*}%
the second $\delta ^{k}$ is the factor needed for $\Delta (Q)$. By (\ref%
{orders}),%
\begin{equation*}
\frac{1}{\left\vert O^{+}(2k,q)\right\vert }-\frac{1}{\left\vert
O^{-}(2k,q)\right\vert }=\frac{1}{q^{k(k-1)}\prod\limits_{i=1}^{k}(q^{2i}-1)}
\end{equation*}%
(note the addition of one more factor in the product). So, again by (\ref%
{orders}),%
\begin{eqnarray*}
\sum_{k=0}^{\left\lfloor n/2\right\rfloor }\sum_{\substack{ Q\in \boldsymbol{%
S}_{n}  \\ \mathrm{rank}Q=2k}}\delta ^{k}q^{-k}\Delta (Q)
&=&\sum_{k=0}^{\left\lfloor n/2\right\rfloor }\delta ^{k}q^{-k} {n \brack 2k}%
_{q} \times \delta ^{k}\frac{q^{k(2k-1)}\prod\limits_{i=1}^{2k}(q^{i}-1)}{%
q^{k(k-1)}\prod\limits_{i=1}^{k}(q^{2i}-1)} \\
&=&\sum_{k=0}^{\left\lfloor n/2\right\rfloor } {n \brack 2k}_{q}
q^{k^{2}-k}\prod\limits_{i=0}^{k-1}(q^{2i+1}-1).
\end{eqnarray*}%
Thus from (\ref{sum2}),%
\begin{equation*}
q^{n(n-1)/2}=\sum_{k=0}^{\left\lfloor n/2\right\rfloor } {n \brack 2k}_{q}
q^{k^{2}-k}\prod\limits_{i=0}^{k-1}(q^{2i+1}-1).
\end{equation*}%
Is that really true? Yes -- it gives the count of the number of symplectic
(skew-symmetric) $n\times n$ matrices over $GF(q)$. From \cite{T}, $%
\left\vert Sp(2k,q)\right\vert =q^{k^{2}}\prod\nolimits_{i=1}^{k}(q^{2i}-1)$%
, so the number of symplectic matrices of rank $2k$ is%
\begin{eqnarray*}
{n \brack 2k}_{q} \times \frac{\left\vert GL(2k,q)\right\vert }{\left\vert
Sp(2k,q)\right\vert } &=& {n \brack 2k}_{q} \frac{q^{k(2k-1)}\prod%
\limits_{i=1}^{2k}(q^{i}-1)}{q^{k^{2}}\prod\limits_{i=1}^{k}(q^{2i}-1)} \\
&=& {n \brack 2k}_{q} q^{k^{2}-k}\prod\limits_{i=0}^{k-1}(q^{2i+1}-1)
\end{eqnarray*}%
\cite[Theorem 13.2.48]{HFF}, and that is to be summed from $k=0$ to $%
\left\lfloor n/2\right\rfloor $ (giving 1 at $k=0$). But the total number of
symplectic $n\times n$ matrices is simply $q^{n(n-1)/2}$.

\section{Minimum Weil-free subgroups of $G$\label{MinWeilfree}}

\label{SecMinWeil-freeSubgr}

Suppose that $H$ is a subgroup of $G$. If $\omega |H$ is also Weil-free,
then $\left\vert H\right\vert \geq q^{n}$. It cannot be that $\left\vert
H\right\vert =q^{n}$, because then $\omega |H$ would just be the sum of all $%
q^{n}$ linear characters of $H$. But that sum is $q^{n}$ at $I$ and $0$ at $%
h\neq I$, whereas $\omega (h)\neq 0$, by (\ref{w(1)}). Thus $\left\vert
H\right\vert \geq 2q^{n}$. We shall show that there are Weil-free subgroups
of order $2q^{n}$. If $\left\vert H\right\vert =2q^{n}$, then $%
H=\left\langle -I\right\rangle \times (H\cap B)$.

Adapting the orbit count in the proof of Theorem \ref{TheoremGWeilfree}, we
find that $2q^n\times$(number of orbits) of such an $H$ is%
\begin{equation*}
2q^{n}+(q^{n}-1)q^{n}+\sum_{x}\sum_{y\neq 0}\left\vert H_{(x,y)}\right\vert
\geq 2q^{n}+(q^{n}-1)q^{n}+q^{n}(q^{n}-1)=2q^{2n}.
\end{equation*}%
So if $H$ is to be Weil-free, each $H_{(x,y)}$ with $y\neq 0$ must be just $%
\left\{ I\right\} $. Again as in the proof of Theorem \ref{TheoremGWeilfree}%
, this means that if $g_{Q}\in H$, with $Q\neq 0$, then $Q$ must have full
rank $n$. So what would work is an $n$-dimensional subspace $W$ of $%
\boldsymbol{S}_{n}$ whose nonzero members are all nonsingular. Then $H\cap B$
would be $\left\{ g_{Q}|Q\in W\right\} $.

To construct $W$, realize $GF(q)^{n}$ as $GF(q^{n})$ and let $\mathrm{tr}$
be the trace function $GF(q^{n})\rightarrow GF(q)$. Then for $\alpha \in
GF(q^{n})$, the function $Q_{\alpha }$ given by $Q_{\alpha }(\zeta )=\mathrm{%
tr}(\alpha \zeta ^{2})$ is a quadratic form on $GF(q^{n})$. The
corresponding bilinear form is $B_{\alpha }(\xi ,\eta )=\mathrm{tr}(\alpha
\xi \eta )$. This is nondegenerate when $\alpha \neq 0$, making $Q_{\alpha }$
nonsingular then. Now let $W=\left\{ Q_{\alpha }|\alpha \in GF(q^{n}\right\}
$.

\subsection{An evaluation for $H$}

Formula (\ref{sig}) for $Q_{a}$ reads%
\begin{equation*}
q^{n}\rho ^{-n}\Delta (Q_{\alpha })=\sum_{\zeta \in GF(q^{n})}\left\{ \frac{%
1+\delta }{2}\psi (Q_{\alpha }(\zeta ))+\frac{1-\delta }{2}\psi (\nu
Q_{\alpha }(\zeta ))\right\} .
\end{equation*}%
Because $\nu \in GF(q)$ and $\rho ^{2}=\delta q$, this becomes%
\begin{equation}
\delta ^{n}\rho ^{n}\Delta (Q_{\alpha })=\sum_{\zeta \in GF(q^{n})}\left\{
\frac{1+\delta }{2}\psi (\mathrm{tr}(\alpha \zeta ^{2}))+\frac{1-\delta }{2}%
\psi (\mathrm{tr}(\nu \alpha \zeta ^{2}))\right\} .  \label{wqn}
\end{equation}%
One has to be careful with the matrix interpretation. Let $\zeta _{1},\ldots
,\zeta _{n}$ be a $GF(q)$-basis of $GF(q^{n})$. Then the matrix for $%
B_{\alpha }$ is $\left[ \mathrm{tr}(\alpha \zeta _{i}\zeta _{j}\right] $.
This can be written as%
\begin{eqnarray*}
\left[ \mathrm{tr}(\alpha \zeta _{i}\zeta _{j}\right] &=&%
\begin{bmatrix}
\zeta _{1} & \zeta _{1}^{q} & \ldots & \zeta _{1}^{q^{n-1}} \\
\zeta _{2} & \zeta _{2}^{q} & \ldots & \zeta _{2}^{q^{n-1}} \\
\vdots & \vdots & \ldots & \vdots \\
\zeta _{n} & \zeta _{n}^{q} & \ldots & \zeta _{n}^{q^{n-1}}%
\end{bmatrix}%
\begin{bmatrix}
\alpha & 0 & \ldots & 0 \\
0 & \alpha ^{q} & \ldots & 0 \\
\vdots & \vdots & \ldots & \vdots \\
0 & 0 & \ldots & \alpha ^{q^{n-1}}%
\end{bmatrix}%
\begin{bmatrix}
\zeta _{1} & \zeta _{2} & \ldots & \zeta _{n} \\
\zeta _{1}^{q} & \zeta _{2}^{q} & \ldots & \zeta _{n}^{q} \\
\vdots & \vdots & \ldots & \vdots \\
\zeta _{1}^{q^{n-1}} & \zeta _{2}^{q^{n-1}} & \ldots & \zeta _{n}^{q^{n-1}}%
\end{bmatrix}
\\
&=&D\times \mathrm{diag}(\alpha ,\alpha ^{q},\ldots ,\alpha
^{q^{n-1}})\times D^{T}.
\end{eqnarray*}%
Taking determinants gives
\begin{equation*}
\det \left[ \mathrm{tr}(\alpha \zeta _{i}\zeta _{j}\right] =(\det
D)^{2}\prod\nolimits_{i=0}^{n-1}\alpha ^{q^{i}}=(\det D)^{2}N(\alpha ),
\end{equation*}%
in which $(\det D)^{2}$ is the discriminant of the extension $%
GF(q^{n})/GF(q) $ and $\prod\nolimits_{i=0}^{n-1}\alpha ^{q^{i}}$ is the
norm $N(\alpha )$ of $\alpha $. Then%
\begin{equation*}
\Delta (Q_{a})=\chi ((\det D)^{2})\chi (N(\alpha )).
\end{equation*}%
Applying the automorphism $\xi \rightarrow \xi ^{q}$ to $D$ cycles its
columns; so $(\det D)^{q}=(-1)^{n-1}\det D$, the sign being that of an $n$%
-cycle. Thus $\chi ((\det D)^{2})=(-1)^{n-1}$: $(\det D)^{2}$ is a square in
$GF(q)$ only when its square-root $\det D$ is in $GF(q)$! For $\chi
(N(\alpha ))$ we have
\begin{equation*}
\chi (N(\alpha ))=N(\alpha )^{\frac{q-1}{2}}=\left( \alpha ^{\frac{q^{n}-1}{%
q-1}}\right) ^{\frac{q-1}{2}}=\alpha ^{\frac{q^{n}-1}{2}}=X(\alpha ),
\end{equation*}%
where $X(\alpha )$ is the quadratic character of $\alpha $ for the field $%
GF(q^{n})$ (we can determine any $\chi (z)$ by reading it in $GF(q)$). All
together, $\Delta (Q_{\alpha })=(-1)^{n-1}X(\alpha )$.

For the right side of (\ref{wqn}), we have that $\sum_{\zeta }\psi (\mathrm{%
tr}(\beta \zeta ^{2}))=X(\beta )P$, where $P$ is the \textquotedblleft $\rho
$\textquotedblright\ for $GF(q^{n})$, by the formulas in (\ref{rho2}) for $%
GF(q^{n})$. Moreover, since $\nu $ is a nonsquare in $GF(q)$, $X(\nu
)=(-1)^{n}$. Therefore (\ref{wqn}) becomes%
\begin{eqnarray*}
\delta ^{n}\rho ^{n}(-1)^{n-1}X(\alpha ) &=&\left\{ \frac{1+\delta }{2}%
X(\alpha )P+\frac{1-\delta }{2}(-1)^{n}X(\alpha )P\right\} \\
&=&X(\alpha )P\left\{ \frac{1+(-1)^{n}}{2}+\delta \frac{1-(-1)^{n}}{2}%
\right\} ,
\end{eqnarray*}%
or%
\begin{equation*}
\delta ^{n}\rho ^{n}(-1)^{n-1}=P\left\{ \frac{1+(-1)^{n}}{2}+\delta \frac{%
1-(-1)^{n}}{2}\right\} .
\end{equation*}%
This simplifies to%
\begin{equation*}
P=(-1)^{n-1}\rho ^{n},
\end{equation*}%
on sorting by the parity of $n$. That is a particular instance of the
Davenport-Hasse theorem on lifted Gauss sums \cite[Theorem 5.14]{LN}.

\end{document}